\documentclass{amsart}

\usepackage{amssymb}
\usepackage{epsfig}
\usepackage{amscd}
\usepackage{amsthm}
\usepackage[noadjust]{cite} 

\newtheorem{thm}{Theorem}[section]
\newtheorem{prop}[thm]{Proposition}
\newtheorem{lem}[thm]{Lemma}

\newtheorem{conj}[thm]{Conjecture} 

\theoremstyle{definition}
\newtheorem{definition}[thm]{Definition}

\theoremstyle{remark}
\newtheorem{remark}[thm]{Remark}

\numberwithin{equation}{section}                                                                                                                                                                                                                                                                                                                                                                                                                                                                                                                                                                                                                                                                                                                                                                                                                                                                                                                                                                                                                                                                                                                                                                                                                                                                                                                                                                                                                                                                                                                                                                                                                                                                                                                                                                                                                                                                                                                                                                                                                                                                                                                                                                                                                                                                                                                                                                                                                                                                                                                                                                                                                                                                                                                                                                                                                                                                                                                                                                                                                                                                                                                                                                                                                                                                                                                                                                                                                                                                                                                                                                                                                                                                                                                                                                                                                                                                                                                                                                                                                                                                                                                                                                                                                                                                                                                                                                                                                                                                                                                                                                                                                                                                                                                                                                                                                                                                                                                                                                                                                                                                                                                                                                                                                                                                                                                                                                                                                                                                                                                                                                                                                                                                                                                                                                                                                                                                                                                                                                                                                                                                                                                                                                                                                                                                                                                                                                                                                                                                                                                                                                                                                                                                                                                                                                                                                                                                                                                                                                                                                                                                                                                                                                                                                                                                                                                                                                                                                                                                                                                                                                                                                                                                                                                                                                                                                                                                                                                                                                                                                                                                                                                                                                                                                                                                                                                                                                                                                                                                                                                                                                                                                                                                                                                                                                                                                                                                                                                                                                                                                                                                                                                                                                                                                                                                                                                                                                                                                                                                                                                                                                                                                                                                                                                                                                                                                                                                                                                                                                                                                                                                                                                                                                                                                                                                                                                                                                                                                                                                                                                                                                                                                                                                                                                                                                                                                                                                                                                                                                                                                                                                                                                                                                                                                                                                                                                                                                                                                                                                                                                                                                                                                                                                                                                                                                                                                                                                                                                                                                                                                                                                                                                                                                                                                                                                                                                                                                                                                                                                                                                                                                                                                                                                                                                                                                                                                                                                                                                                                                                                                                                                                                                                                                                                                                                                                                                                                                                                                                                                                                                                                                                                                                                                                                                                                                                                                                                                                                                                                                                                                                                                                                                                                                                                                                                                                                                                                                                                                                                                                                                                                                                                                                                                                                                                                                                                                                                                                                                                                                                                                                                                                                                                                                                                                                                                                                                                                                                                                                                                                                                                                                                                                                                                                                                                                                                                                                                                                                                                                                                                                                                                                                                                                                                                                                                                                                                                                                                                                                                                                                                                                                                                                                                                                                                                                                                                                                                                                                                                                                                                                                                                                                                                                                                                                                                                                                                                                                                                                                                                                                                                                                                                                                                                                                                                                                                                                                                                                                                                                                                                                                                                                                                                                                                                                                                                                                                                                                                                                                                                                                                                                                                                                                                                                                                                                                                                                                                                                                                                                                                                                                                                                                                                                                                                                                                                                                                                                                                                                                                                                                                                                                                                                                                                                                                                                                                                                                                                                                                                                                                                                                                                                                                                                                                                                                                                                                                                                                                                                                                                                                                                                                                                                                                                                                                                                                                                                                                                                                                                                                                                                                                                                                                                                                                                                                                                                                                                                                                                                                                                                                                                                                                                                                                                                                                                                                                                                                                                                                                                                                                                                                                                                                                                                                                                                                                                                                                                                                                                                                                                                                                                                                                                                                                                                                                                                                                                                                                                                                                                                                                                                                                                                                                                                                                                                                                                                                                                                                                                                                                                                                                                                                                                                                                                                                                                                                                                                                                                                                                                                                                                                                                                                                                                                                                                                                                                                                                                                                                                                                                                                                                                                                                                                                                                                                                                                                                                                                                                                                                                                                                                                                                                                                                                                                                                                                                                                                                                                                                                                                                                                                                                                                                                                                                                                                                                                                                                                                                                                                                                                                                                                                                                                                                                                                                                                                                                                                                                                                                                                                                                                                                                                                                                                                                                                                                                                                                                                                                                                                                                                                                                                                                                                                                                                                                                                                                                                                                                                                                   
\newcommand{\R}{\mathbf{R}}

\begin{document}
\title{Cobordism invariants of the moduli space of stable pairs}

\author{Junliang Shen}
\address{Departement Mathematik\\ETH Z\"urich\\R\"amistrasse 101\\8092 Z\"urich\\Switzerland}
\email{junliang.shen@math.ethz.ch}


\begin{abstract}
For a quasi-projective scheme $M$ which carries a perfect obstruction theory, we construct the virtual cobordism class of $M$ which is the universal virtual fundamental class. If $M$ is projective, we prove that the corresponding Chern numbers  of the virtual cobordism class are given by integrals of the Chern classes of the virtual tangent bundle. Further, we study cobordism invariants of the moduli space of stable pairs introduced by Pandharipande--Thomas. Rationality of the partition function is conjectured together with a functional equation, which can be regarded as a generalization of the rationality and $q^{-1} \leftrightarrow q$ symmetry of the Calabi--Yau case. We prove rationality for nonsingular projective toric 3-folds by the theory of descendents.
\end{abstract}

\maketitle
\setcounter{tocdepth}{1}
\tableofcontents

\section{Introduction}
\subsection{Algebraic cobordism and virtual classes}
Let $M$ be a quasi-projective scheme which carries a perfect obstruction theory. The virtual fundamental class $[M]^{\mathrm{vir}}$ can be constructed in the expected Chow group $A_{\mathrm{vir.dim}}(M)$ by the methods of Li--Tian \cite{LT} and Behrend--Fantechi \cite{BF}, \cite{Beh}. Virtual fundamental classes play crucial roles in the study of enumerative geometry.

Our first result (Theorem \ref{thm1}) in this paper is to construct the virtual fundamental class in algebraic cobordism, the universal oriented Borel--Moore homology theory \cite{LM}, \cite{LP}. More precisely, if a quasi-projective scheme $M$ carries a perfect obstruction theory, there exists a cobordism class $[M]^{\mathrm{vir}}_{\Omega_\ast}\in \Omega_{\mathrm{vir.dim}}(M)$. Moreover, this universal virtual class yields virtual fundamental classes in other oriented BM homology theories (for example, the virtual cycle in Chow theory and the virtual structure sheaf in $K$-theory \cite{Lee}).

\subsection{Chern numbers}
It is well known that cobordism classes over one point are governed by Chern numbers. Fixing $d > 0$, one can choose a basis $\{{v_I}\}$ of $\Omega_d(\mathrm{pt})_\mathbb{Q}:= \Omega_d(\mathrm{pt})\otimes \mathbb{Q}$ such that the class $[X \rightarrow \mathrm{pt}]$ can be expressed as
\[
 \sum_{\substack{I=(i_1,i_2,\dots, i_d)\\ \sum_{k}{ki_k}=d}}{\int_{X}{c_1(X)^{i_1}c_2(X)^{i_2}\cdots c_d(X)^{i_d}}\cdot v_I},
 \]
where $X$ is nonsingular and of dimension $d$. For any class in $\Omega_\ast (\mathrm{pt})_{\mathbb{Q}}$, Chern numbers are defined to be the coefficients of such $v_I$.

Now assume $M$ is projective. We can push forward the virtual cobordism class $[M]^{\mathrm{vir}}_{\Omega_\ast}$ along the projection $\pi_{M}:M\rightarrow \mathrm{pt}$. The following theorem concerns the Chern numbers~of 
\[
{\pi_M}_\ast [M]^{\mathrm{vir}}_{\Omega_\ast} \in \Omega_{\mathrm{vir.dim}}(\mathrm{pt}).
\]

\begin{thm}\label{thmchernnumber}
The Chern numbers of the cobordism class ${\pi_M}_\ast [M]^{\mathrm{vir}}_{\Omega_\ast}$ are given by the virtual Chern numbers
\[\int_{[M]^{\mathrm{vir}}}{c_1(T^{\mathrm{vir}})^{i_1}c_2(T^{\mathrm{vir}})^{i_2}\cdots c_{\textup{vir.dim}}(T^\mathrm{vir})^{i_\textup{vir.dim}}}.
\]
Here $[M]^{\mathrm{vir}} \in A_{\mathrm{vir.dim}}(M)$ is the (Chow) virtual fundamental class and $T^{\mathrm{vir}} \in K^0(M)$ is the virtual tangent bundle obtained from the perfect obstruction theory.
\end{thm}

In \cite{CFK} Theorem 4.6.4, Ciocan--Fontanine and Kapranov proved that the virtual Chern numbers of a $[0,1]$-manifold can be realized as the Chern numbers of a ($\mathbb{Z}$-coefficient) cobordism class.\footnote{Over a point, the algebraic cobordism ring $\Omega_\ast({\textup{pt}})$ is canonically isomorphic to the complex cobordism ring.} Our results show that this class actually is obtained from the virtual cobordism class we constructed. This gives an algebro-geometric interpretation of their results. Lowery and Sch\"{u}rg \cite{DAG} constructed fundamental classes in cobordism for quasi-smooth derived schemes, which are expected to give the same classes as ours when the perfect obstruction theory is from the quasi-smooth structure. 

A review of perfect obstruction theories, Chow virtual classes and the precise definition of virtual tangent bundles can be found in Section 1. 

\subsection{Moduli space of stable pairs and descendents} 
Let $X$ be a nonsingular projective 3-fold, and let $\beta \in H_2(X,\mathbb{Z})$ be a nonzero class. We study here the moduli space of stable pairs
\[      [\mathcal{O}_X\xrightarrow{s} F] \in P_n(X, \beta)
\]
where $F$ is a pure sheaf supported on a Cohen--Macaulay subcurve of $X$, $s$ is a morphism with 0-dimensional cokernel, and 
\[  \chi(F)=n, [F]=\beta.
\]
The projective scheme $P_n(X, \beta)$ carries a perfect obstruction theory obtained from the deformation theory of complexes in the derived category and the virtual dimension (vir.dim.) is equal to $ \int_{\beta}c_1(X)$. See \cite{PT} for details.

Since $P_n(X, \beta)$ is a fine moduli space, there exists a universal sheaf
\[
\mathbb{F} \rightarrow X \times P_n(X, \beta).
\]
For a stable pair $[\mathcal{O}_X\rightarrow F] \in P_n(X,\beta)$, the restriction of $\mathbb{F}$ to the fiber
\[
X \times [\mathcal{O}_X \rightarrow F]\subset X\times P_n(X, \beta) 
\] is canonically isomorphic to $F$. Let
\[
\pi_X: X \times P_n(X, \beta) \rightarrow X,\]
\[\pi_P: X \times P_n(X, \beta) \rightarrow P_n(X, \beta)
\] be the projections onto the first and second factors. Since $\mathbb{F}$ has a finite resolution by locally free sheaves, the Chern character of the universal sheaf $\mathbb{F}$ is well-defined.

The descendents $\tau_i(\gamma)$ with $\gamma \in H^\ast(X, \mathbb{Z})$ are defined to be the operators on the homology of $P_n(X, \beta)$
\begin{equation*}
{\pi_P}_\ast\big{(}\pi_X^\ast(\gamma)\cdot \mathrm{ch}_{2+i}(\mathbb{F})\cap \pi_P^\ast(~\cdot~)\big{)}: H_\ast (P_n(X, \beta), \mathbb{Q}) \rightarrow H_\ast (P_n(X , \beta),\mathbb{Q}).
\end{equation*}

For nonzero $\beta \in H_2(X, \mathbb{Z})$ and arbitrary $\gamma_i\in H^{\ast}(X, \mathbb{Z})$, define the stable pairs descendent invariants by
\begin{align*}
 \bigg{<}\prod_{j=1}^{m} \tau_{i_j}(\gamma_j) \bigg{>}_{n,\beta}^X &=
\int_{[P_n(X, \beta)]^{\mathrm{vir}}} \prod_{j=1}^{m} \tau_{i_j}(\gamma_j) \\
&= \int_{[P_n(X, \beta)]} \prod_{j=1}^{m} \tau_{i_j}(\gamma_j) ([P_n(X,\beta)]^{\mathrm{vir}}).
\end{align*}
Then the partition function of descendents is 
\begin{equation*}
Z_\beta^{X}\bigg{(}  \prod_{j=1}^{m} \tau_{i_j}(\gamma_j) \bigg{)} = \sum_n \bigg{<}    \prod_{j=1}^{m} \tau_{i_j}(\gamma_j)   \bigg{>}_{n, \beta}^{X} q^n.
\end{equation*}

Descendent invariants of stable pairs have been studied in \cite{stationary}, \cite{PPrationality}, \cite{rationality}. We have the following theorem relating virtual Chern numbers to descendent invariants. 
\begin{thm}\label{thmdescendents}
For any nonsingular projective 3-fold $X$, the integral
\[\int_{[P_n(X, \beta)]^{\mathrm{vir}}}{c_1(T^{\mathrm{vir}})^{i_1}c_2(T^{\mathrm{vir}})^{i_2}\cdots c_d(T^\mathrm{vir})^{i_d}}
\] can be expressed in terms of descendent invariants. Here $d$ is equal to the virtual dimension $\int_{\beta}c_1(X)$.
\end{thm}

Hence we can express the Chern numbers of ${\pi}_\ast [P_n(X, \beta)]^{\mathrm{vir}}_{\Omega_\ast}$ in terms of descendents. 

Note that the same proof shows that this statement holds for Donaldson--Thomas theory. In Gromov--Witten theory, the virtual tangent bundle is closely related to the Hodge bundle on the moduli space of stable maps, and expressions of the Chern classes of the virtual tangent bundle in terms of descendents were studied in \cite{faber} and \cite{Givental}.

 \subsection{Cobordism invariants}
 Let $\pi: P_n(X, \beta) \rightarrow \mathrm{pt}$ be the projection and $d$ be the virtual dimension $\int_\beta c_1(X)$. We obtain the cobordism invariant $[P_{n,\beta}]$ by pushing forward the cobordism virtual fundamental class, $i.e.$
 \[
 [P_{n,\beta}]: = \pi_\ast [P_n(X, \beta)]_{\Omega_\ast}^{\mathrm{vir}} \in \Omega_d(\mathrm{pt}).
  \]
By Theorem \ref{thmchernnumber} $[P_{n,\beta}]$ are deformation invariants. Note that when $X$ is a Calabi--Yau 3-fold, we have $d=0$. Then the invariants $[P_{n,\beta}] \in \mathbb{Z}$ are just the degree of the Chow virtual cycle, which are known as the Pandharipande--Thomas invariants, see Definition 2.16 of \cite{PT}.

Now we define the partition function of cobordism invariants of stable pairs to be
\begin{equation*}
Z_{\Omega_\ast,\beta}^{X}(q):=\sum_n{[P_{n,\beta}] q^n} \in \Omega_\ast(\mathrm{pt})\otimes {\mathbb{Q}}((q)).
\end{equation*}
If we choose a basis $\{ v_I\}$ of $\Omega_d(\mathrm{pt})_\mathbb{Q}$, the partition function can be written as
\begin{equation} \label{partition}
Z_{\Omega_\ast,\beta}^{X}(q) = \sum_{I}f_{I}(q) \cdot v_I.
\end{equation}
Here $f_I(q) \in \mathbb{Q}((q))$ is a Laurent series in $q$. It is clear that the right-hand side is a finite sum.

\begin{conj} \label{mainconj}
The partition function $Z_{\Omega_\ast,\beta}^{X}(q)$ satisfies the following properties:
\begin{enumerate}
\item{\bf \textup{(Rationality)}} $Z_{\Omega_\ast,\beta}^{X}(q)$ is a rational function in $q$. More precisely, in the expression of (\ref{partition}), each function $f_I(q)$ is the Laurent expansion of a rational function in $q$.

\item{\bf \textup{(Functional equation)}}  The rational function $Z_{\Omega_\ast,\beta}^{X}(q)$ satisfies the functional equation
\begin{equation*}
Z_{\Omega_\ast,\beta}^{X}(q^{-1}) = q^{-d} Z_{\Omega_\ast,\beta}^{X}(q),
\end{equation*}
$i.e.$, each function $f_I(q)$ satisfies this functional equation.
\end{enumerate}
\end{conj}

When $X$ is Calabi--Yau, it is clear that our conjectures specialize to Conjecture 3.2 of \cite{PT} which has been proven in \cite{Bri} (Theorem 1.1). This is part of the conjectures equating Gromov--Witten theory with sheaf counting theory for Calabi--Yau 3-folds \cite{MNOP1} \cite{PT}. See also \cite{MNOP2}, \cite{MOOP}, \cite{BPS}, \cite{toric}, \cite{quintic}, \cite{Toda1}, \cite{Toda2}, and \cite{ST} for related works. However, contrary to the case of Calabi--Yau 3-folds, calculations by computer show that the simplest form of DT/Pairs correspondence 
\[
\Big{(}\sum_n{[P_{n,\beta}] q^n}\Big{)}\Big{(}\sum_{m} [I_{n,0}] q^m \Big{)} = \sum_{n} [I_{n,\beta}] q^n\]
doesn't hold for any nonsingular projective 3-fold $X$ in cobordism. Here $[I_{n,0}
] \in \mathbb{Z}$ and $[I_{n,\beta}] \in \Omega_\ast (\mathrm{pt})$ are the corresponding invariants obtained from the Donaldson--Thomas moduli spaces $I_n(X, \beta)$. It would be interesting to find out the correct form of the cobordism DT/Pairs correspondence (if it exists).

Specializations of the cobordism class like the virtual elliptic genus, the virtual $\chi_{-y}$ genus and the virtual Euler characteristic were studied in \cite{FG}. All these invariants for stable pairs theory are expected to satisfy Conjecture \ref{mainconj}.

The rationality of toric cases is proven in this paper.
\begin{thm} \label{toric}
Conjecture \ref{mainconj} (1) is true if $X$ is a nonsingular projective toric 3-fold.
\end{thm}

\subsection{Plan of the paper}
We start with a review of perfect obstruction theories and the deformation to the normal cone in Section 1. The construction of the virtual cobordism class is given in Section 1.3. In Section 2, we introduce twisted Chow theory. The isomorphism between algebraic cobordism and twisted Chow theory gives formulas for the Chern numbers of the virtual cobordism class. Hence Theorem \ref{thmchernnumber} is proved. In Section 3 and 4, we study cobordism invariants of the moduli space of stable pairs. By a Grothendieck--Riemann--Roch calculation, we express the Chern characters of the virtual tangent bundle in terms of descendents in Section 3, which proves Theorem \ref{thmdescendents}. As a consequence, we prove Theorem \ref{toric} in Section 4 by the rationality of descendents. Finally we calculate several examples via localization to support Conjecture \ref{mainconj} (2).

\subsection{Notation} We work over the field $\mathbb{C}$.  All schemes in this paper are assumed to be quasi-projective. We use $\mathrm{pt}$ to denote the scheme $\mathrm {Spec}(\mathbb{C})$. The functors of Chow theory and algebraic cobordism theory are denoted by $A_\ast$ and $\Omega_\ast$. We work with $\mathbb{Z}$-coefficients unless stated otherwise. For a nonsingular scheme $X$, the cobordism fundamental class $[X\xrightarrow{\mathrm{id}}X] \in \Omega_{\mathrm{dim}(X)}(X)$ is denoted by ${\bf 1}_X$. When $E$ is a vector bundle (or coherent sheaf) on $X$, we use the notation $E|_U$ to denote the pull back of $E$ to the subscheme $U$ of $X$.

\subsection{Acknowledgement}
First, I would like to thank my advisor Rahul Pandharipande for suggesting this topic, sharing his ideas and many insightful conversations. His sharing of communications with Marc Levine about algebraic cobordism was very helpful. Discussions with Andrew Morrison, Goerg Oberdieck and Timo Sch{\"u}rg about the virtual fundamental class and the moduli space of stable pairs play an important role. Thanks to Ionut Ciocan-Fontanine, Andrew Morrison, Christoph Schiessl, and Qizheng Yin for their valuable comments, and to the referee for suggesting several improvements. Also, I would like to thank Zhihong Tan for his help with MATLAB.

The author was supported by grant ERC-2012-AdG-320368-MCSK in the group of Pandharipande at ETH Z{\"u}rich.

\section{Virtual fundamental class in algebraic cobordism}
\subsection{Perfect obstruction theory}
We work on a quasi-projective scheme $M$ and assume it is globally embedded into a nonsingular scheme $Y$ with the corresponding ideal sheaf $I$. A perfect obstruction theory on $M$ consists of the following data:
\begin{enumerate}
\item[(1)] A two term complex of locally free sheaves $E^\bullet= [E^{-1}\rightarrow E^0]$ on $M$.
\item[(2)] A morphism $\phi: E^\bullet \rightarrow L_M^\bullet$ in the derived category of coherent sheaves to the cotangent complex $L_M^\bullet$ satisfying that $\phi$ induces an isomorphism in cohomology in degree 0 and a surjection in cohomology in degree -1.
\end{enumerate}

Clearly, only the information of the truncated complex $L_M^{\bullet{\ge -1}}(=[L^{-1}/\mathrm{Im}(L^{-2}) \rightarrow L^{0}\rightarrow \cdots]$) is  used in the data of a perfect obstruction theory. Hence we can use the explicit representative
\[ L_M^{\bullet{\ge -1}} =  [I/I^2 \rightarrow \Omega_Y|_M].
\]

A morphism $\phi^\bullet:E^\bullet \rightarrow [I/I^2 \rightarrow \Omega_Y|_M]$ in the derived category means there exists a two term complex $[G^{-1}\rightarrow G^0]$ of coherent sheaves such that $G^\bullet \xrightarrow{\varphi_1}E^\bullet$ and $G^\bullet \xrightarrow{\varphi_2}[I/I^2 \rightarrow \Omega_Y|_M]$ are both morphisms of complexes with $\varphi_1$ being a quasi-isomorphism and $\varphi_2$ inducing an isomorphism in the cohomology $h^0$ and epimorphism in $h^{-1}$. Since $M$ is quasi-projective, any coherent sheaf on $M$ can be written as the quotient of a locally free sheaf. It implies that we may choose both $G^0$ and $G^{-1}$ to be locally free. Hence we can use the following data as a perfect obstruction theory on a quasi-projective scheme $M$:

\begin{equation}\label{diag1}
\begin{CD}
E^{-1} @>>> E^0 \\
@VVV         @VVV\\
I/I^2     @>>> \Omega_Y|_M
\end{CD}
\end{equation} 
where  $h^0(\phi)$ is an isomorphism and $h^{-1}(\phi)$ is an epimorphism.

Equivalently, such data can also be written in the language of abelian cones (or linear spaces in the sense of \cite{Siebert})
\[
\Phi_\bullet:[T_Y|_M \rightarrow N_{M/Y} ] \rightarrow [E_0 \rightarrow E_1]
\]
where $E_i=(E^{-i})^{\vee}(i=0,1)$ and $N_{M/Y}= \mathrm{Spec}(\mathrm{Sym}(I/I^2))$ is the corresponding normal sheaf. The morphism $\Phi_\bullet$ induces an isomorphism on the cohomology $H^0$ and a closed embedding on $H^1$.

Hence we get the following short exact sequence of abelian cones from (\ref{diag1})
\begin{equation}\label{es1}
0\rightarrow T_Y|_M \rightarrow N_{Y/M}\times_M E_0\rightarrow C(Q) \rightarrow 0
\end{equation}
where $C(Q)$ is a closed sub-abelian cone of $E_1$. 

The normal cone $C_{M/Y} = \mathrm{Spec}( \bigoplus_{n\ge 0}I^n/I^{n+1} )$ is a sub-cone of the normal sheaf $N_{M/Y}$, and the cone $C_{M/Y} \times_M E_0$ is a $T_Y|_M$-invariant sub-cone of the abelian cone $N_{M/Y}\times_ME_0$. Moreover, $T_Y|_M$ acts on $C_{M/Y}\times _M E_0$ freely and fiberwise (see \cite{BF}). Hence the quotient cone 
\[D^{\mathrm{vir}}:= \frac{C_{M/Y}\times _M E_0}{T_Y|_M}  \]
is a closed subscheme of the vector bundle $E_1$. It is easily checked that $\mathrm{dim}(D^{\mathrm{vir}})= \mathrm{rk}(E_0)$. The Chow virtual class defined in \cite{LT} and \cite{BF} is exactly the refined intersection of $D^{\mathrm{vir}}$ with the zero section of the vector bundle $E_1$, $i.e.$
\begin{equation*}
[M]^{\mathrm{vir}} = 0_{E_1}^{!}[D^{\mathrm{vir}}]\in A_{\mathrm{rk}(E_0)-\mathrm{rk}(E_1)}(M).
\end{equation*}

Finally we define the virtual tangent bundle to be the $K$-theory class 
\[
T_M^{\mathrm{vir}} := [E_0] - [E_1] \in K^0(M).
\] It is clear that $T_M^{\mathrm{vir}}$ only depends on the perfect obstruction theory.\footnote{It is independent of the choice of global resolution of the tangent-obstruction bundles $E^\bullet$.} The virtual dimension is equal to $\mathrm{rk}(E_0) - \mathrm{rk}(E_1)$.

\subsection{Deformation to the normal cone}
With the assumption and notation above, we will construct a morphism
\begin{equation*}
\sigma: \Omega_\ast(Y) \rightarrow \Omega_{\ast}(C_{M/Y})
\end{equation*}
by the technique of deformation to the normal cone (see Chapter 5 of \cite{fulton}, 3.2.1 and 6.5.2 of \cite{LM}). Then from this map we obtain a canonical class $\sigma({\bf 1}_Y) \in \Omega_{\textup{dim}(Y)}(C_{M/Y})$ which can serve as the fundamental class of the normal cone $C_{M/Y}$, although $C_{M/Y}$ may not be smooth generally. The idea is to construct a family of embeddings $M \hookrightarrow Y_t$ parametrized by $t \in \mathbb{P}^1$, such that for $t \neq 0$, the embedding is the given embedding of $M$ in $Y$, and for $t = 0$ one has the zero section embedding of $M$ in $C_{M/Y}$. We review the construction below.

Consider the embedding of $M$ in $Y\times \mathbb{P}^1$ which is the composition of $M\hookrightarrow Y$ and the embedding of $Y$ in $Y \times \mathbb{P}^1$ at $0\in \mathbb{P}^1$. Define
\[W:=\mathrm{ Bl}_{M\times\{0\}}(Y\times \mathbb{P}^1).
\]
There is a  projection
\[\pi: W \rightarrow \mathbb{P}^1
\]
such that 
\begin{enumerate}
\item $\pi^{-1}(t)=Y$ when $t\neq0$;
\item $\pi^{-1}(0)= \mathrm{Bl}_{M}(Y) \cup P(C_{M/Y}\oplus O_M)$.
\end{enumerate} Here $P(C_{M/Y}\oplus O_M)$ denotes the projective cone associated to the affine cone $C_{M/Y} \oplus O_M$. 
Now we denote by $W^{\circ}$ the open subset of $W$ obtained by removing the component $\mathrm{Bl}_{M}(N)$ from the fiber over 0. Therefore the fiber over 0 in $W^{\circ}$ is $C_{M/Y}$.

Consider the following diagram
\begin{equation}\label{dfnc}
\begin{CD}
\Omega_{k+1}(C_{M/Y}) @>{j_\ast}>> \Omega_{k+1}(W^{\circ}) @>>> \Omega_{k+1}(Y \times \mathbb{A}^1) @>>> 0 \\
@.                                                    @V{j^{\ast}}VV                  @A{\mathrm{pr}^{\ast}}AA\\
                                 @.                   \Omega_{k}(C_{M/Y})                 @<{\sigma}<<  \Omega_k(Y).
\end{CD}
\end{equation}
The top row is the localization exact sequence (Theorem 3.2.7 of \cite{LM}). Since the divisor $C_{M/Y}$ moves in $W^{\circ}$, the composition map $j^\ast j_\ast$ is zero. So we get the following map
\[
\sigma:\Omega_k(Y) \cong \Omega_{k+1}(Y\times \mathbb{A}^1) \rightarrow \Omega_{k}(C_{M/Y}).
\]

\subsection{Virtual class in $\Omega_{\mathrm{vir.dim}}(M)$}
Recall that since $T_Y|_M$ acts on $C_{M/Y}\times_XE_0$ freely and $D^{\mathrm{vir}}$ is the quotient cone, the cone $C_{M/Y} \times_ME_0$ is a $p^{\ast}(T_Y|_M)$-principal homogeneous space over $D^{\mathrm{vir}}$, where $p: D^{\mathrm{vir}} \rightarrow M$. By the homotopy invariance property (Theorem 3.6.3 of \cite{LM}), there is an isomorphism:
\[\Omega_\ast(D^{\mathrm{vir}}) \cong \Omega_{\ast+\mathrm{dim}(Y)}(C_{M/Y}\times_ME_0).
\]
Hence we get the following chain of maps 
\begin{align*}   
\Omega_\ast(Y)  \xrightarrow{~~~\sigma~~~}  \Omega_{\ast}(C_{M/Y})  \xrightarrow{~~\simeq~~}   \Omega_{\ast + \mathrm{rk}(E_0)}(C_{M/Y}\times_ME_0)  \simeq  \Omega_{\ast+\mathrm{rk}(E_0)-\mathrm{dim}(Y)}(D^{\mathrm{vir}})\\
\xrightarrow{~~e_\ast~~} \Omega_{\ast  +\mathrm{rk}(E_0)-\mathrm{dim}(Y)}(E_1) \xrightarrow{~~0_{E_1}^!~~} \Omega_{\ast+\mathrm{rk}(E_0)-\mathrm{rk}(E_1)-\mathrm{dim}(Y)}(M).
\end{align*}
Here $e_\ast$ is induced by the closed embedding $e: D^{\textup{vir}} \rightarrow E_1$.
\begin{thm}\label{thm1}
Let $M$ be a scheme which is embedded into a nonsingular scheme $Y$. Given a perfect obstruction theory $E^\bullet=[E^{-1}\rightarrow E^0]\rightarrow L_M^\bullet$ on $M$, there is a morphism of algebraic cobordism groups
\begin{equation}
\sigma_{M/Y}: \Omega_{\ast}(Y) \rightarrow \Omega_{\ast+\mathrm{rk}(E_0)-\mathrm{rk}(E_1)-\mathrm{dim}(Y)}(M).
\end{equation}
Moreover, the image $\sigma_{M/Y}({\bf 1}_Y) \in \Omega_{\mathrm{rk}(E_0)-\mathrm{rk}(E_1)}(M)$ is independent of the choice of embedding $M\hookrightarrow Y$, and representative of the perfect obstruction theory.
\end{thm}
The map $\sigma_{M/Y}$ is just given by the chain above. We will prove the independence below. Note that in the $\mathbb{Q}$-coefficient cobordism group, the independence is a direct conclusion of Theorem \ref{relation} (see Remark \ref{easy}). The following part is to treat the $\mathbb{Z}$-coefficient case which may be skipped.\footnote{Readers can jump to Definition \ref{skip} directly.} 

Now we prove the independence. Since two different embeddings are dominated by the diagonal embedding, we only need to treat the following case.

Let $i:M \rightarrow Y$ and $i': M\rightarrow Y'$ be two embeddings with a smooth map $\pi: Y' \rightarrow Y$ such that $i=\pi \circ i'$. Assume there are two representatives of a perfect obstruction theory: $[T_Y|_M \rightarrow N_{M/Y}] \rightarrow [E_0 \rightarrow E_1]$ and $[T_{Y'}|_M \rightarrow N_{M/Y'}] \rightarrow [F_0 \rightarrow F_1]$. Then there is a two term complex of vector bundles $[G_0 \rightarrow G_1]$ such that $E_\bullet \rightarrow G_\bullet$ and $F_\bullet \rightarrow G_\bullet$ are both quasi-isomorphisms. It suffices to show:
\begin{enumerate}
\item the representatives $[T_Y|_M \rightarrow N_{M/Y}] \rightarrow E_\bullet$ and $[T_{Y'}|_M \rightarrow N_{M/Y'}] \rightarrow E_\bullet$ yield the same virtual class;
\item the representatives $[T_{Y'}|_M \rightarrow N_{M/Y'}] \rightarrow E_\bullet$ and $[T_{Y'}|_M \rightarrow N_{M/Y'}] \rightarrow G_\bullet$ yield the same virtual class.
\end{enumerate}

To simplify the proof, we introduce here an equivalent description of the map $\sigma_{M/Y}$ which was used in \cite{GP}. Consider the following Cartesian diagrams (see \cite{GP} Section 3, (12))
\[
\begin{CD}
T_Y|_M @>>>  C_{M/Y} \times_M E_0\\
@VVV    @VVV\\
M @>>> D^{\mathrm{vir}} \\
@V{\textup{id}}VV  @VVV\\
M @>{0_{E_1}}>>E_1.
\end{CD}
\]
It is easy to see that $\sigma_{M/Y}$ can be obtained by the following chain of maps
\begin{align*}
\Omega_\ast(Y)  \xrightarrow{~~~\sigma~~~}  \Omega_{\ast}(C_{M/Y})  \xrightarrow{~~\simeq~~}  \Omega_{\ast + \mathrm{rk}(E_0)}(C_{M/Y}\times_ME_0)  \xrightarrow{0_{E_1}^!}  \\
\Omega_{\ast +\mathrm{rk}(E_0)-\mathrm{rk}(E_1)}(T_Y|_M) \xrightarrow{0_{T_Y}^\ast} \Omega_{\ast+\mathrm{rk}(E_0)-\mathrm{rk}(E_1)-\mathrm{dim}(Y)}(M).
\end{align*}

\bigskip
Assume the image $\sigma({\bf 1}_Y)$ (resp. $\sigma({\bf 1}_{Y'})$) is denoted by $\alpha_{M/Y}$ (resp. $\alpha_{M/Y'}$). And assume $\pi': C_{M/Y'}\rightarrow C_{M/Y}$ is the map induced by the smooth map $\pi: Y' \rightarrow Y$. 
\begin{lem}\label{lem1}
The following diagram commutes
\begin{equation*}
\begin{CD}
\Omega_\ast(Y)  @>{\sigma_{Y}}>> \Omega_\ast(C_{M/Y})\\
@V{\pi^\ast}VV                           @V{\pi'^\ast}VV\\
\Omega_{\ast +r}(Y')  @>{\sigma_{Y'}}>>\Omega_{\ast +r}(C_{M/Y'}),
\end{CD}
\end{equation*}
where $r$ is the rank of the relative tangent bundle $T_{Y'/Y}$. In particular, we have $ \pi'^\ast \alpha_{M/Y} = \alpha_{M/Y'}$.
\end{lem}
\begin{proof}[Proof of Lemma \ref{lem1}]
Suppose $W^{\circ}_{Y}$ (resp. $W^{\circ}_{Y'}$) is the total space of the deformation to the normal cone constructed by the embedding $i$ (resp. $i'$). The map $\pi: Y' \rightarrow Y$ extends to a smooth morphism ${\pi}_{W}: W_{Y'}^{\circ} \rightarrow{W_{Y}^{\circ}}$ inducing the natural map $C_{M/Y'} \rightarrow C_{M/Y}$. Take an element
\[
\xi \in \Omega_\ast(Y) \cong \Omega_{\ast +1}(Y \times \mathbb{A}^1),
\]
and let $\xi' \in \Omega_{\ast+1}(W_{Y}^{\circ})$ be a lifting of $\xi$ (from (\ref{dfnc})). By the Cartesian diagram 
\[
\begin{CD}
Y'\times \mathbb{A}^1@>>>W_{Y'}^\circ  \\
@V{\pi \times \mathrm{id}}VV   @V{\pi_W}VV\\
Y \times \mathbb{A}^1@>>>W_{Y}^\circ 
\end{CD}
\]
the class $\pi_{W}^\ast {\xi}' $ is a lifting of $\pi^\ast \xi$. Hence
\begin{align*}
{\pi'}^{\ast} {\sigma}_{Y}(\xi) & =  {{\pi'}^{\ast}}(\xi' |_{C_{M/Y}})    \\
                              & = (\pi_W^\ast \xi' )|_{C_{M/Y'}}\\
                               & = \sigma_{Y'}(\pi^\ast \xi)
\end{align*}
 where $\xi'|_{C_{M/Y}}$ means the pull back of $\xi'$ by the embedding $C_{M/Y'} \hookrightarrow W_{Y'}^\circ$.
\end{proof}

Now we turn to prove (1) and (2).
\begin{proof}[Proof of (1)]
We have the following Cartesian diagrams:
\[\begin{CD}
T_{Y'}|_M  @>>> C_{M/Y'}\times_ME_0 @>{p_1}>> C_{M/Y'}\\
@V{f}VV @VgVV  @V{\pi'}VV   \\
T_Y|_M @>>>  C_{M/Y}\times_ME_0 @>{p_2}>> C_{M/Y}  \\
@VVV @VVV\\
M @>{0_{E_1}}>> E_1.
\end{CD}
\]
According to the second description of $\sigma_{M/Y}$ above, it suffices to prove
\[
0_{E_1}^!p_1^\ast \alpha_{M/Y'} = f^! 0_{E_1}^! p_2^\ast \alpha_{M/Y}.
\]
In fact, by Lemma \ref{lem1} and intersection theory of algebraic cobordism (Chapter 6 of \cite{LM}), we have
\begin{align*}
0_{E_1}^!p_1^\ast \alpha_{M/Y'} & = 0_{E_1}^!p_1^\ast \pi'^\ast \alpha_{M/Y} \\
                                                    & = 0_{E_1}^!g^\ast p_2^\ast \alpha_{M/Y} \\
                                                    & = f^\ast 0_{E_1}^! p_2^\ast \alpha_{M/Y}. \qedhere
\end{align*}
\end{proof}

\begin{proof} [Proof of (2)]
Consider the following Cartesian diagrams:
\[
\begin{CD}
T_{Y'}|_M @>>>  C_{M/Y'}\times_ME_0 @>{j'}>> C_{M/Y'}\times_M G_0 \\
@VVV  @VVV @VVV\\
M @>{0_{E_1}}>> E_1 @>{j}>> G_1.
\end{CD}
\]
Here both $0_{E_1}$ and $j$ are local complete intersection (l.c.i.) morphisms, therefore $(j\circ0_{E_1})^!=0_{E_1}^! j^!$ (Theorem 6.6.6 of \cite{LM}). 

Assume the projections $C_{M/Y'}\times_ME_0 \rightarrow C_{M/Y'}$ and $C_{M/Y'}\times_MG_0 \rightarrow C_{M/Y'}$ are denoted by $p_E$ and $p_G$. Then it suffices to prove
\[
j^!p_G^\ast \alpha_{M/Y'} = p_E^\ast \alpha_{M/Y'}.
\]
It holds since 
\[
j^!p_G^\ast \alpha_{M/Y'} = j'^!p_G^\ast \alpha_{M/Y'} = p_E^\ast \alpha_{M/Y'}. \qedhere
\]
\end{proof}

\begin{definition}\label{skip}
With the same notation as above, we define
\[
[M]_{\Omega_\ast}^{\mathrm{vir}} := \sigma_{M/Y}({\bf 1}_Y) \in \Omega_{\mathrm{vir.dim}}(M)
\] to be the virtual cobordism class of $M$ with respect to the perfect obstruction theory $E^\bullet \rightarrow L_M^\bullet$.
\end{definition}

\begin{remark}
By definition, the virtual cobordism class we construct yields the virtual fundamental class in Chow theory and the virtual structure sheaf in $K$-theory by the universality of the functor $\Omega_\ast$.
\end{remark}

\section{Twisted Chow theory and Chern numbers}
\subsection{Todd classes and twisted Chow theory}
In order to obtain the formulas for Chern numbers of ${\pi_M}_\ast {[M]_{\Omega_\ast}^{\mathrm{vir}}}$ (Theorem \ref{thmchernnumber}), we introduce here twisted Chow theory. We briefly review the construction. More details can be found in 4.1 and 7.4 of \cite{LM}.

Let $\mathbb{Z}[{\bf t}]:= \mathbb{Z}[t_1,t_2, \dots , t_n,\dots]$ be the graded ring of polynomials with integral coefficients in the variables $t_i (i>0)$ of degree $i$. Starting with Chow theory $A_\ast$, we consider the oriented Borel--Moore homology theory $S\mapsto A_\ast(S) \otimes_\mathbb{Z}\mathbb{Z}[{\bf t}]$ by extension of scalars. 

Define the universal inverse Todd class operator of a line bundle $L \rightarrow S$ to be the operator on $A_\ast(S)[{\bf t}]$ given by 
\[\mathrm{Td}_{\bf t}^{-1}(L) = \sum_{i=0}^{\infty}c_1(L)^it_i : A_\ast(S)[{\bf t}] \rightarrow A_\ast(S)[{\bf t}].
\]
Note that by the definition of the first Chern class operator, there are only finitely many non-zero terms in this power series. We can extend such operators to all vector bundles, $i.e.$ there is a homomorphism
\[\mathrm{Td}_{\bf t}^{-1}: K^0(S) \rightarrow \mathrm{Aut}(A_\ast(S)[{\bf t}]).
\]
Let $E$ be any vector bundle over $S$, then $\mathrm{Td}_{\bf t}^{-1}(E)$ can be expanded as
\begin{equation}\label{CF}
\mathrm{Td}_{\bf t}^{-1}(E)=\sum_{I = (n_1,\cdots, n_d,\cdots)}c_I(E)t_1^{n_1}\dots t_d^{n_d}\dots.
\end{equation}
The $c_I(E)$ are the Conner--Floyed Chern class operators. By the splitting principle, it can be checked easily that the $\mathbb{Q}$-vector space generated by $c_I(E)$ $ (|I|=\Sigma kn_k = m)$ is same as the $\mathbb{Q}$-vector space generated by $\prod_{I=(n_1,\dots, n_d,\dots)}c_i(E)^{n_i}$ ($|I|=m$). 

Now we can twist the oriented Borel--Moore homology theory $A_\ast[{\bf t}]$ by the universal inverse Todd class operator to get a new theory $A_\ast[{\bf t}]^{({\bf t})}$. The construction is the following:
\begin{enumerate}
\item $A_\ast(S)[{\bf t}]^{(\bf t)}:= A_\ast(S)[{\bf t}]$ for any scheme $S$.
\item $f_\ast^{(\bf t)} = f_\ast$.
\item For any l.c.i. morphism $f:S_1 \rightarrow S_2$, choose a factorization of $f$ as $f=q\circ i$ with $i:S_1\rightarrow P$ a regular embedding and $q:P\rightarrow S_2$ a smooth morphism. The virtual normal bundle $N_f$ is defined to be the $K$-theory class $[N_i]-[i^\ast T_q]$ which is independent of the choice of the factorization. \footnote{$N_i$ is the normal bundle of $i$ and $T_p$ is the relative tangent bundle of $p$.} Then
\[f_{(\bf t)}^\ast = \mathrm{Td}_{\bf t}^{-1}([N_f])\circ f^\ast.
\]
\end{enumerate}
The new theory $A_\ast [{\bf t}]^{(\bf t)}$ is still an oriented Borel--Moore homology theory. The following theorem linking algebraic cobordism theory $\Omega_\ast$ and twisted Chow theory $A_\ast [{\bf t}]^{(\bf t)}$ is crucial.

\begin{thm}[{\cite{LM}}] \label{LM}
There is a canonical morphism
\[\vartheta : \Omega_\ast \rightarrow A_\ast[{\bf t}]^{({\bf t})}
\]
which induces an isomorphism after tensoring with $\mathbb{Q}$.
\end{thm}

\subsection{Virtual class in $A_\ast (M) [{\bf t}]_\mathbb{Q}^{({\bf t})}$   }
We turn back to the virtual class and keep the same notation as in Section 1. By Theorem \ref{LM}, we identify $\Omega_\ast (S)_{\mathbb{Q}}$ with $A_\ast (S) [{\bf t}]_\mathbb{Q}^{({\bf t})}$ for any scheme $S$. 

If $M$ is a nonsingular (or l.c.i.) scheme, we have the following formula for the fundamental class under the isomorphism between  $\Omega_\ast (M)_{\mathbb{Q}}$ and $A_\ast (M) [{\bf t}]_\mathbb{Q}^{({\bf t})}$
\begin{equation} \label{formula1}
{\bf 1}_M = \mathrm{Td}_{\bf t}^{-1}(-[T_M])[M].
\end{equation}
It is obtained immediately by pulling back the fundamental class ${\bf 1}_{ \mathrm{pt}}$ along the projection $M \rightarrow \mathrm{pt}$. We have a more general formula which can be regarded as the virtual version of (\ref{formula1}).

\begin{thm} \label{relation}
Let $M$ be a scheme, and $E^\bullet= [E^{-1} \rightarrow E^0] \rightarrow L_M^\bullet$ be a perfect obstruction theory. Then under the isomorphism between $\Omega_\ast \otimes \mathbb{Q}$ and $A_\ast[\bf t]_\mathbb{Q}^{({\bf t})}$, the following holds:
\begin{equation*}
[M]_{\Omega_\ast}^{\mathrm{vir}} = \mathrm{Td}_{\bf t}^{-1}(-[T_M^{\mathrm{vir}}])[M]^{\mathrm{vir}}.
\end{equation*}
Here $[M]^{\mathrm{vir}} \in A_{\mathrm{vir.dim}}(M)$ is the Chow virtual class obtained by the same perfect obstruction theory.
\end{thm}

\begin{remark}\label{easy}
It is seen easily that the cobordism virtual class in $\Omega_\ast \otimes \mathbb{Q}$ does not rely on the choice of embedding and representative of the perfect obstruction theory from Theorem \ref{relation}, since the Chow virtual class does not.
\end{remark}

The strategy of the proof is to chase the chain of maps from which we defined the virtual cobordism class:
\begin{align*}   
\Omega_\ast(Y)_{\mathbb{Q}}  \xrightarrow{~~~\sigma~~~}  \Omega_{\ast}(C_{M/Y})_{\mathbb{Q}}\xrightarrow{~~\simeq~~}   \Omega_{\ast + \mathrm{rk}(E_0)}(C_{M/Y}\times_ME_0)_{\mathbb{Q}}  \simeq  \Omega_{\ast+\mathrm{rk}(E_0)-\mathrm{dim}(Y)}(D^{\mathrm{vir}})_{\mathbb{Q}}\\
\xrightarrow{~~\phantom{\sigma}~~} \Omega_{\ast +\mathrm{rk}(E_0)-\mathrm{dim}(Y) }(E_1)_{\mathbb{Q}} \xrightarrow{~~0_{E_1}^!~~} \Omega_{\ast+\mathrm{rk}(E_0)-\mathrm{rk}(E_1)-\mathrm{dim}(Y)}(M)_{\mathbb{Q}}.
\end{align*}
The proof is devided by 5 steps as follows.

\bigskip
\noindent\textit{Step 1.} Let $\pi_1:C_{M/Y}\rightarrow M$ be the projection. We claim that 
\[\alpha_{M/Y}= \mathrm{Td}_{\bf t}^{-1}(-[\pi_1^\ast T_Y|_M])[C_{M/Y}] \in \Omega_\ast(C_{M/Y})_\mathbb{Q}(= A_\ast(C_{M/Y})[{\bf t}]^{({\bf t})}_\mathbb{Q}).
\]

\begin{proof}
Consider  \[\pi_W: W=\mathrm{Bl}_{M\times \{0\}}(Y\times \mathbb{P}^1) \rightarrow Y.
\]
We restrict the vector bundle $\pi_W^\ast T_Y$ on $W$ to the open subset $W^\circ = W \setminus \mathrm{Bl}_M(Y)$. For convenience, this vector bundle is also denoted by $\pi_W^\ast T_Y$. Now we have the following diagram
\[\begin{CD}
\Omega_{\mathrm{dim}(Y)+1}(C_{M/Y})_\mathbb{Q} @>{j_\ast}>> \Omega_{\mathrm{dim}(Y)+1}(W^{\circ})_\mathbb{Q} @>{u}>> \Omega_{\mathrm{dim}(Y)+1}(Y \times \mathbb{A}^1)_{\mathbb{Q}} @>>> 0 \\
@.                                                    @V{j^{\ast}}VV                  @A{\mathrm{pr}^{\ast}}AA\\
                                 @.                   \Omega_{\mathrm{dim}(Y)}(C_{M/Y})_\mathbb{Q}                 @<{\sigma}<<  \Omega_{\mathrm{dim}(Y)}(Y)_\mathbb{Q}.
\end{CD}
\]
By (\ref{formula1}), we have ${\bf 1}_Y = \mathrm{Td}_{\bf t}^{-1}(-[T_Y])[Y] \in \Omega_{\mathrm{dim}(Y)}(Y)_\mathbb{Q}$. Then
\[ {\bf 1}_{Y\times \mathbb{A}^1} = \mathrm{pr}^\ast{\bf 1_Y} = \mathrm{Td}_{\bf t}^{-1}(-[\mathrm{pr}^\ast T_Y])[Y\times \mathbb{A}^1].
\]
Consider the class
\[\theta := \mathrm{Td}_{\bf t}^{-1}(-[\pi_W^\ast T_Y])[W^\circ] \in \Omega_{\mathrm{dim}(Y)+1}(W^\circ)_\mathbb{Q}.
\]It is clear that $u(\theta)= {\bf 1}_{Y\times \mathbb{A}^1}$. Hence
\[  \alpha_{M/Y}= \sigma({\bf 1}_Y) = j^\ast\theta = \mathrm{Td}_{\bf t}^{-1}(-[(\pi_W^\ast T_Y)|_{C_{M/Y}}])[C_{M/Y}],
\] and obviously $(\pi_W^\ast T_Y)|_{C_{M/Y}} = \pi_1^\ast (T_Y|_M)$.
\end{proof}
\noindent \textit{Step 2.} Let $\pi_2: C_{M/Y}\times_ME_0 \rightarrow C_{M/Y}$ be the projection. We have\[
\beta( := \pi_2^\ast \alpha_{M/Y} )= \mathrm{Td}_{\bf t}^{-1}(-[T_Y]-[E_0])[C_{M/Y}\times_ME_0].
\] (Note that we still use $T_Y$ and $E_0$ to represent the corresponding bundles pulled back from $M$ for notational convenience, and we will use such convention in Steps 3,4,5 below.)
\begin{proof}
It follows from the definition of pull back in $A_\ast[{\bf t}]^{({\bf t})}_\mathbb{Q}$, since $[N_{\pi_2}] = -[T_{\pi_2}] = -[E_0] \in K^0(C_{M/Y}\times E_0)$.
\end{proof}

\noindent \textit{Step 3.} Let $\pi_3:C_{M/Y}\times_ME_0\rightarrow D^{\mathrm{vir}}$ be the quotient map. Since \[\pi_3^\ast:\Omega_\ast (D^{\mathrm{vir}}) \rightarrow \Omega_{\ast + \mathrm{dim}(Y)}(C_{M/Y}\times_M E_0)\] induces an isomorphism, there exists $\gamma \in \Omega_{\mathrm{rk}(E_0)}(D^{\mathrm{vir}})$ such that $\pi_3^\ast\gamma = \beta$. We claim that\[
\gamma = \mathrm{Td}_{\bf t}^{-1}(-[E_0])[D^{\mathrm{vir}}] \in \Omega_{\mathrm{rk}(E_0)}(D^{\mathrm{vir}})_\mathbb{Q}.
\]
\begin{proof}
Since $\pi_3^\ast$ induces an isomorphism, it suffices to prove \[\pi_3^\ast(\mathrm{Td}_{\bf t}^{-1}(-[E_0])[D^{\mathrm{vir}}]) = \beta.\]
In fact, we have $[N_{\pi_3}] = -[T_Y] \in K^0(C_{M/Y}\times_ME_0)$ by the definition of $D^{\textup{vir}}$. Therefore
\[
\pi_3^\ast(\mathrm{Td}_{\bf t}^{-1}(-[E_0])[D^{\mathrm{vir}}])=\mathrm{Td}_{\bf t}^{-1}(-[T_Y])\mathrm{Td}_{\bf t}^{-1}(-[E_0])\pi_3^{\ast}[D^{\mathrm{vir}}] = \beta.\qedhere\]
\end{proof}

\noindent \textit{Step 4.} Pushing forward along the embedding $i:D^{\mathrm{vir}} \rightarrow E_1$, we get
\[\gamma'(:=i_\ast\gamma) = \mathrm{Td}_{\bf t}^{-1}(-[E_0])[D^{\mathrm{vir}}] \in \Omega_{\mathrm{rk}(E_0)}(E_1).
\]

\noindent \textit{Step 5.} We finish the proof of Theorem \ref{relation}, $i.e.$,
\[[M]_{\Omega_\ast}^{\mathrm{vir}} = \mathrm{Td}_{\bf t}^{-1}([E_1]-[E_0])[M]^{\mathrm{vir}} \in \Omega_{\mathrm{rk}(E_0)-\mathrm{rk}(E_1)}(M)_\mathbb{Q}.
\]
\begin{proof}
Let $\pi_4: E_1 \rightarrow M$ be the projection. It suffices to prove that\[
\pi_4^\ast(\mathrm{Td}_{\bf t}^{-1}([E_1]-[E_0])[M]^{\mathrm{vir}}) = \gamma'.
\]
Actually we have 
\begin{align*}
\pi_4^\ast(\mathrm{Td}_{\bf t}^{-1}([E_1]-[E_0])[M]^{\mathrm{vir}}) & = \mathrm{Td}_{\bf t}^{-1}(-[E_1])\mathrm{Td}_{\bf t}^{-1}([E_1]-[E_0])\pi_4^\ast[M]^{\mathrm{vir}}\\
& = \mathrm{Td}_{\bf t}^{-1}(-[E_0])[D^{\mathrm{vir}}].\qedhere
\end{align*}
\end{proof}

\subsection{Chern numbers}
The cobordism group $\Omega_\ast(\mathrm{pt})$ has the structure of a graded ring. The morphism in Theorem \ref{LM} induces an isomorphism between the rings $\Omega_\ast(\mathrm{pt})_\mathbb{Q}$ and $\mathbb{Q}[{\bf t}]$. Hence every class can be expressed in terms of homogeneous polynomials in the variables $t_1,\dots , t_k, \dots$.

If $M$ is a nonsingular projective scheme, we consider the expression of $[M \rightarrow \mathrm{pt}]$ as homogeneous polynomials. By pushing forward the fundamental class of $M$ to one point and formula~(\ref{formula1}), we get
\begin{equation} \label{formula}
[M\rightarrow \mathrm{pt}]= \sum_{\sum_{k}{kn_k}=\mathrm{dim}(M)}\big{(}{\int_{[M]}c_I(-[T_M])}\big{)}{\bf t}^I,
\end{equation}
where $I=(n_1,\dots, n_d, \dots)$ and ${\bf t}^I=t_1^{n_1}\dots t_d^{n_d}\dots$. (recall that $c_I$ are the Conner--Floyed Chern class operators. See (\ref{CF}).)
Actually, formula (\ref{formula}) explains the reason that cobordism classes over one point are governed by Chern numbers.

In general, when $M$ is a projective scheme carrying a perfect obstruction theory, similarly we push forward the virtual cobordism class along the projection $\pi_M: M \rightarrow \mathrm{pt}$. By Theorem \ref{relation}, we have
\[
{\pi_M}_\ast{[M]_{\Omega_\ast}^{\mathrm{vir}}}= \sum_{\sum_{k}{kn_k}=\mathrm{vir.dim}}\big{(}{\int_{[M]^{\mathrm{vir}}}c_I(-[T_M^{\mathrm{vir}}])}\big{)}{\bf t}^I,
\]
where $I=(n_1,\dots, n_d, \dots)$ and ${\bf t}^I=t_1^{n_1}\dots t_d^{n_d}\dots$. This gives the proof of Theorem \ref{thmchernnumber}.

\section{Stable pairs: descendents, generalized descendents and Chern characters}
We use the same notation as in Section 0.3 and 0.4. For notational convenience, sometimes we just write $P(X)$ to denote the moduli space $P_n(X, \beta)$.

\subsection{Descendents and generalized descendents}
Similar to the descendents defined in Section 0.3, we define the generalized descendents $\tau_{i,j}(\gamma)$ as the following operators
\[{\pi_P}_\ast\big{(}\pi_X^\ast(\gamma)\cdot \mathrm{ch}_{2+i}(\mathbb{F})\cdot \mathrm{ch}_{2+j}(\mathbb{F})\cap \pi_P^\ast(~\cdot~)\big{)}: H_\ast(P(X), \mathbb{Q}) \rightarrow H_\ast(P(X),\mathbb{Q}).
\]
Since $\pi_P: P(X)\times X \rightarrow P(X)$ is an l.c.i and proper morphism, we have the Gysin homomorphism 
\begin{equation*}
{\pi_P}_\ast: H^n(P(X) \times X) \rightarrow H^{n-6}(P(X)).
\end{equation*} (Actually in our case, the push forward is just given by the K{\"u}nneth decomposition.) Hence descendents and generalized descendents lie in $H^{\ast}(P(X))$, $i.e.$,
\begin{equation*}
\tau_{i}(\gamma)= {\pi_P}_\ast(\pi_X^\ast\gamma \cdot \mathrm{ch}_{2+i}(\mathbb{F})) \in H^\ast(P(X)),
\end{equation*}
\begin{equation*}
\tau_{i,j}(\gamma)= {\pi_P}_\ast(\pi_X^\ast\gamma \cdot \mathrm{ch}_{2+i}(\mathbb{F})\cdot \mathrm{ch}_{2+j}(\mathbb{F})) \in H^\ast(P(X)).
\end{equation*}

We show that generalized descendents can be expressed in terms of descendents.

Consider the diagonal embedding $\delta: X \rightarrow X \times X$. For $\gamma \in H^\ast(X)$, the class $\delta_\ast\gamma$ has K{\"u}nneth decomposition $\sum_{i}{u_i \otimes v_i}$ where $u_i, v_i \in H^\ast(X)$. In other words, let $q_i: X\times X \rightarrow X$ be the projection onto the $i$-th factor (i=1,2), then $\delta_\ast\gamma = \sum_{i}q_1^\ast u_i \cdot q_2^\ast v_i$.

We look at the space $P(X) \times X\times X$. From the two projections $p_i: P(X)\times X\times X \rightarrow P(X)\times X$ (i=1,2), we obtain two sheaves $\mathbb{F}_1$, $\mathbb{F}_2$ on $P(X)\times X\times X$ by pulling back the universal sheaf $\mathbb{F}$ via $p_i$. We have the following commutative diagram:
\begin{equation}
\begin{CD}
P(X)\times X @>{\Delta}>> P(X)\times X \times X @>\pi>> P(X)\\
@V{\pi_X}VV    @V\pi_{X\times X}VV\\
X @>{\delta}>> X\times X.
\end{CD}
\end{equation}

\begin{lem} \label{lem2}
The following formula holds
\[
\tau_{i, j}(\gamma) = {\pi}_\ast\big{(}\pi_{X\times X}^\ast(\delta_\ast\gamma)\cdot \mathrm{ch}_{2+i}(\mathbb{F}_1)\cdot \mathrm{ch}_{2+j}(\mathbb{F}_2)\big{)}
\]
where $\gamma \in H^\ast(X)$.
\end{lem}
\begin{proof}
The strategy here is to use the projection formula. By definition
\begin{align*}
\tau_{i,j}(\gamma)& =  \pi_\ast \Delta_\ast ( \Delta^\ast(\mathrm{ch}_{2+i}(\mathbb{F}_1)\cdot \mathrm{ch}_{2+j}(\mathbb{F}_2))\cdot \pi_X^\ast\gamma)\\
& =    \pi_\ast (\mathrm{ch}_{2+i}(\mathbb{F}_1)\cdot \mathrm{ch}_{2+j}(\mathbb{F}_2)\cdot \Delta_\ast\pi_X^\ast\gamma)\\
& =   \pi_\ast (\mathrm{ch}_{2+i}(\mathbb{F}_1)\cdot \mathrm{ch}_{2+j}(\mathbb{F}_2)\cdot \pi_{X\times X}^\ast(\delta_\ast\gamma)). \qedhere
\end{align*}
\end{proof}

\begin{prop} \label{prop1}
With the notation as above, we have 
\[\tau_{k_1,k_2}(\gamma) =  \sum_{i}{\tau_{k_1}(u_i)\cdot \tau_{k_2}(v_i)}.
\]
\end{prop}
\begin{proof}
We have the following commutative diagram
\[
\begin{CD}
P(X)\times X @>{\triangle}>> P(X)\times X \times X @>{p_2}>> P(X)\times X \\
@.                                                  @V{p_1}VV                                     @V{\pi_2}VV \\
                        @.                         P(X) \times X           @>{\pi_1}>>   X.
\end{CD}
\]
Again, we use the projective formula. By the diagram above and Lemma \ref{lem2}, we have
\begin{align*}
\tau_{k_1, k_2}(\gamma) &=  {\pi_1}_\ast{p_1}_\ast (\pi_{X \times X}^\ast (\delta_\ast \gamma)\cdot p_1^\ast \mathrm{ch}_{2+k_1}(\mathbb{F})\cdot p_2^\ast \mathrm{ch}_{2+k_2}(\mathbb{F})) \\
&=    {\pi_1}_\ast{p_1}_\ast ((\sum_{i} p_1^\ast \pi_1^\ast u_i \cdot p_2^\ast \pi_2 ^\ast v_i)\cdot p_1^\ast \mathrm{ch}_{2+k_1}(\mathbb{F})\cdot p_2^\ast \mathrm{ch}_{2+k_2}(\mathbb{F})) \\
&= \sum_i {\pi_1}_\ast{p_1}_\ast (p_1^\ast (\pi_1^\ast u_i \cdot \mathrm{ch}_{2+k_1}(\mathbb{F}))\cdot p_2^\ast (\pi_2^\ast v_i \cdot \mathrm{ch}_{2+k_2}(\mathbb{F})))\\
&= \sum_i {\pi_1}_\ast(\pi_1^\ast u_i \cdot \mathrm{ch}_{2+k_1}(\mathbb{F}) \cdot {p_1}_\ast p_2^\ast(\pi_2^\ast v_i \cdot \mathrm{ch}_{2+k_2}(\mathbb{F})))\\
&= \sum_i {\pi_1}_\ast(\pi_1^\ast u_i \cdot \mathrm{ch}_{2+k_1}(\mathbb{F}) \cdot \pi_1^\ast {\pi_2}_\ast(\pi_2^\ast v_i \cdot \mathrm{ch}_{2+k_2}(\mathbb{F})))\\
&= \sum_i {\pi_1}_\ast(\pi_1^\ast u_i \cdot \mathrm{ch}_{2+k_1}(\mathbb{F}) )\cdot {\pi_2}_\ast(\pi_2^\ast v_i \cdot \mathrm{ch}_{2+k_2}(\mathbb{F}))\\
 &= \sum_{i}{\tau_{k_1}(u_i)\cdot \tau_{k_2}(v_i)}. \qedhere
 \end{align*}
\end{proof}
This shows that generalized descendents can be written in terms of descendents.

\subsection{Virtual tangent bundle}
We study the virtual tangent bundle $T^{\mathrm{vir}}$ of $P_n(X, \beta)$. See \cite{PT} for more details. Let
\[
\mathbb{I}^\bullet =[\mathcal{O}_{P(X)\times X} \rightarrow \mathbb{F}] 
\] be the universal stable pair on $P(X) \times X$. There is a perfect obstruction theory on $P(X)$ such that the tangent-obstruction spaces at a pair $I^\bullet \in P(X)$ are then given by the traceless extension groups $\mathrm{Ext}^1 (I^\bullet , I^\bullet)_0$ and $\mathrm{Ext}^2 (I^\bullet , I^\bullet)_0$, $i.e.$,
\[
T^{\mathrm{vir}}_{[I^\bullet]} = \mathrm{Ext}^1 (I^\bullet , I^\bullet)_0 - \mathrm{Ext}^2 (I^\bullet , I^\bullet)_0.
\]
The element $R{\pi_P}_\ast R\mathcal{H}om(\mathbb{I}^\bullet, \mathbb{I}^\bullet)_0$ in the bounded derived category of coherent sheaves $D^b(P(X))$ represents a class in $K^{0}(P(X))$ which is denoted by $[R{\pi_P}_\ast R\mathcal{H}om(\mathbb{I}^\bullet, \mathbb{I}^\bullet)_0]$. Then by the description above, we have
\begin{equation}\label{formula2}
-T^{\mathrm{vir}} = [E_1] - [E_0] =  [R{\pi_P}_\ast R\mathcal{H}om(\mathbb{I}^\bullet, \mathbb{I}^\bullet)_0] \in K^{0}(P(X)).
\end{equation}
And clearly by definition
\begin{equation}\label{formula3}
[R{\pi_P}_\ast R\mathcal{H}om(\mathbb{I}^\bullet, \mathbb{I}^\bullet)_0] =  [R{\pi_P}_\ast R\mathcal{H}om(\mathbb{I}^\bullet, \mathbb{I}^\bullet)] - [R{\pi_P}_\ast\mathcal{O}_{P(X)\times X}].
\end{equation}

\subsection{Chern characters}
Applying the Grothendieck--Riemann--Roch theorem, we show that the Chern class $c_k(T^{\mathrm{vir}})$ can be expressed in terms of generalized descendents. Then combining with Proposition \ref{prop1}, we finish the proof of Theorem \ref{thmdescendents}.

Equivalently, we prove the following proposition. 

\begin{prop}
The Chern character $\mathrm{ch}_k (- T^{\mathrm{vir}})$  can be expressed in terms of generalized descendents.
\end{prop}

\begin{proof}
By (\ref{formula2}) and (\ref{formula3}), we have
\begin{align*}
\mathrm{ch}_k(-T^{\mathrm{vir}}) & =  \mathrm{ch}_k([R{\pi_P}_\ast R\mathcal{H}om(\mathbb{I}^\bullet, \mathbb{I}^\bullet)]) - \mathrm{ch}_k([R{\pi_P}_\ast \mathcal{O}_{P(X)\times X}])\\
                                 & =  \mathrm{ch}_k([R{\pi_P}_\ast(\mathbb{I}^\bullet\otimes^{L} (\mathbb{I}^\bullet)^{\vee})]) - \mathrm{ch}_k([R{\pi_P}_\ast \mathcal{O}_{P(X)\times X}]).
\end{align*}    
Here $\otimes^{L}$ means the tensor product in the derived category.  

By applying the Grothendieck--Riemann--Roch formula (see \cite{Fulton2}) to the morphism
\[   \pi_P:P(X)\times X \rightarrow P(X),
\] 
we get 
\begin{align*}
 \mathrm{ch}([R{\pi_P}_\ast(\mathbb{I}^\bullet\otimes^{L} (\mathbb{I}^\bullet)^{\vee})]) & = 
 {\pi_P}_\ast(\mathrm{ch}(\mathbb{I}^\bullet\otimes^{L} (\mathbb{I}^\bullet)^{\vee}) \cdot \mathrm{Td}(\pi_X^\ast T_X))\\
 & =  {\pi_P}_\ast(\mathrm{ch}(\mathbb{I}^\bullet) \cdot \mathrm{ch}((\mathbb{I}^\bullet)^{\vee}) \cdot \pi_X^\ast \mathrm{Td}(T_X)).
\end{align*}
Take the degree $k$ part $(k>0)$, we see that $\mathrm{ch}_k([R{\pi_P}_\ast(\mathbb{I}^\bullet\otimes^{L}(\mathbb{I}^\bullet)^{\vee})]) $ can be expressed as
\begin{equation} \label{Express}
\sum_{2i+2j+\mathrm{deg}(\gamma_{ij})=2k+6}a_{ij}\cdot {\pi_P}_\ast(\mathrm{ch}_i(\mathbb{F})\cdot \mathrm{ch}_j(\mathbb{F})\cdot\pi_X^\ast\gamma_{ij})
\end{equation}
where $a_{ij} \in \mathbb{Q}$ and $\gamma_{ij}\in H^\ast(X)$. 
Similarly, we have
\[\mathrm{ch}(R{\pi_P}_\ast \mathcal{O}_{P(X)\times X}) = {\pi_P}_\ast(\pi_X^\ast \mathrm{Td}(T_X)).\] 
Note that $\mathrm{ch}_k(R{\pi_P}_\ast \mathcal{O}_{P(X)\times X}) = 0$ if $k >0 $ for dimension reasons. Hence $\mathrm{ch}_k(-T^{\mathrm{vir}})$ can be represented as in (\ref{Express}) when $k>0$. 
\end{proof}

\section{Cobordism invariants: rationality and functional equation}
\subsection{Rationality of the partition functions of descendents} The rationality of the partition functions of descendents was conjectured by Pandharipande--Thomas \cite{vertex} and proved for nonsingular toric 3-folds in \cite{PPrationality}.
\begin{conj}[\cite{vertex}] \label{conj1}
The partition function $Z_\beta^{X}\big{(}  \prod_{j=1}^{m} \tau_{i_j}(\gamma_j) \big{)}$ is the Laurent expansion of a rational function in $q$.
\end{conj}

\begin{thm} [\cite{PPrationality}] \label{PP}
Conjecture \ref{conj1} is true for all nonsingular projective toric 3-folds.
\end{thm}
Actually, a stronger statement is proved in \cite{PPrationality}. They showed that the rationality holds even for the $\bf T$-equivariant partition function of descendents, where $\bf T$ is a 3-dimensional algebraic torus acting on the nonsingular toric 3-fold $X$ ($X$ need not be compact). See Theorem 1 of \cite{PPrationality}.

\subsection{Proof of Theorem \ref{toric}} 
We choose the basis $v_I$ as in Section 0.2. Then under this basis, the functions $f_I(q)$ (see (\ref{partition})) are just the partition functions of Chern numbers, $i.e.$,
\[ 
   f_I(q) = \sum_{n} c_n^{I} q^n
\]
where
\[ c_n^{I} :=  \int_{[P_n(X, \beta)]^{\mathrm{vir}}}{c_1(T^{\mathrm{vir}})^{i_1}c_2(T^{\mathrm{vir}})^{i_2}\cdots c_n(T^\mathrm{vir})^{i_n}}, I=(i_1,i_2, \dots). \]

Hence it is clear that Theorem \ref{toric} is a consequence of Theorem \ref{thmdescendents} and Theorem \ref{PP}.

\begin{remark}\label{mystery} We should mention here that although the virtual Chern numbers can be expressed in terms of descendent invariants, it is still a mystery why the functional equations hold. From the point of view of the Gromov--Witten/Pairs correspondence \cite{toric} \cite{quintic}, the partitions function of descendents in Gromov--Witten theory and stable pairs are related by the following equation
\[
(-q)^{-d/2}Z_\beta^{X}\Big{(}  \prod_{j=1}^{m} \tau_{i_j}(\gamma_j) \Big{)}_{P} = 
(-iu)^d Z_\beta^{X}\Big{(} {\prod_{j=1}^{m} \overline{\tau_{i_j}(\gamma_j)}} \Big{)}_{\textup{GW}}
\]
under the variable change $q= - e^{iu}$. Here the correspondence rule
\[
\prod_{j=1}^{m} \tau_{i_j}(\gamma_j) \mapsto {\prod_{j=1}^{m} \overline{\tau_{i_j}(\gamma_j)}}
\]
is given by a matrix with coefficients in $\mathbb{C}$ (see \cite{quintic} Section 0.3). When the virtual dimension $d$ is even, the condition that $Z_\beta^{X}\Big{(}  \prod_{j=1}^{m} \tau_{i_j}(\gamma_j) \Big{)}_{P}$ satisfy the functional equation $f(1/q) = q^{-d}f(q)$ implies $Z_\beta^{X}\Big{(} {\prod_{j=1}^{m} \overline{\tau_{i_j}(\gamma_j)}} \Big{)}_{\textup{GW}}\in \mathbb{Q}((u))$, which is not true in general. When $d$ is odd, even the invariants of primary field insertions do not satisfy our functional equation.\footnote{The partition functions of primary fields satisfy the functional equation $f(1/q) = (-q)^{d}f(q)$ by the GW/Pairs correspondence. The computation in Section 4.4 (b) provides evidence for the functional equation we conjectured when $d$ is odd.} Hence properties of normal stable pairs invariants (primary fields and descendents) may not explain why Conjecture \ref{mainconj} (2) holds. 
\end{remark}

\subsection{Toric calculation} For a nonsingular toric 3-fold $X$ acted upon by a 3-dimensional complex torus $\bf T$, the $\bf T$-character of the virtual tangent bundle $T^{\mathrm{vir}}$ can be written explicitly (Theorem 2 of \cite{vertex}). For convenience, we consider the case that the fixed loci of the torus action are isolated points on $P_n(X, \beta)$. Then by \cite{GP} we have
\begin{align*}
c_n^{I} &=  \int_{[P_n(X, \beta)]^{\mathrm{vir}}}{c_1(T^{\mathrm{vir}})^{i_1}c_2(T^{\mathrm{vir}})^{i_2}\cdots c_n(T^\mathrm{vir})^{i_n}}\\
           &=\sum_{\textup{fixed point }Q}\frac{c_1({\bf T}^{\mathrm{vir}})^{i_1}c_2({\bf T}^{\mathrm{vir}})^{i_2}\cdots c_n({\bf T}^\mathrm{vir})^{i_n}}{e({\bf T}^{\mathrm{vir}})}.
\end{align*}
Here ${\bf T}^{\mathrm{vir}}$ is the $\bf{T}$-character of $T^{\mathrm{vir}}|_Q$ (the details of the localization formula of the virtual class $[P_n(X, \beta)]^{\mathrm{vir}}$ can be found in Section 4 of \cite{vertex}). This gives us a concrete way to compute $c_n^I$ term by term. We will use this method and computer to calculate a few terms of the series $\sum_{n} c_n^{I} q^n$ and see how the rational function looks like. We will also check that these rational functions we obtain satisfy the functional equation we conjectured. This provides the numerical evidence for Conjecture \ref{mainconj}. 

However, it should be mentioned that we get these rational functions experimentally. Complete proofs of the close formulas have not yet been shown for most cases.\footnote{For the special case when $\beta$ is a line in the nonsingular toric 3-fold $X$, we may give the strict proof as follows. First, we can bound the degree of the denominator and numerator of each partition function of descendents by the degree one vertex (Section 6.4 of \cite{PT} ). Since for fixed $X$ and $\beta$, there are only finite cases of descendent invariants by dimension reasons, the degree of the denominator and numerator of each partition function of Chern numbers are also bounded. Hence a few terms determine the whole function.}

\subsection{Irreducible curve class} First we give examples when $X$ is a nonsingular projective 3-fold and $\beta$ is a line.

\bigskip
\noindent (a). $X=\mathbb{P}^1 \times \mathbb{P}^2$, $\beta = [\mathbb{P}^1]\times \textup{pt}$. So we have $d(=\mathrm{vir.dim}) = 2$.

(1). I=(2),
\[\sum_{n}{c_n^I q^n}\approx \frac{9q}{(1+q)^4} \Big{(}  1+4q+ 30q^2 +4q^3 +q^4 \Big{)}.\]

(2). I=(0,1),
\[\sum_{n}{c_n^I q^n} \approx \frac{3q}{(1+q)^4} \Big{(}  1+4q+ 42q^2 +4q^3 +q^4 \Big{)}.\]
Both rational functions satisfy the functional equation
\[
 f(1/q) = q^{-2}f(q).
\]

\bigskip
\noindent (b). $X=\mathbb{P}^1 \times \mathbb{P}^2$, $\beta = \textup{pt}\times [L]$, where $[L]$ is a line in $\mathbb{P}^2$. Then $d = 3$.

(1). I=(3),
\[\sum_{n}{c_n^I q^n} \approx \frac{6q}{(1+q)^5} \Big{(}  9+47q+ 73q^2 +550q^3 +73q^4+47q^5 + 9q^6 \Big{)}.\]

(2). I=(1,1),
\[\sum_{I}{c_n^I q^n} \approx \frac{24q}{(1+q)^5} \Big{(}  1+6q+ 12q^2 +74q^3 +12q^4+6q^5+q^6 \Big{)}.\]

(3). I=(0,0,1),
\[\sum_{I}{c_n^I q^n} \approx \frac{6q}{(1+q)^5} \Big{(}  1+5q+ 7q^2 +86q^3 +7q^4+5q^5+q^6 \Big{)}.\]
All rational functions satisfy the functional equation
\[
 f(1/q) = q^{-3}f(q).
\]

\bigskip
\noindent (c). $X = \mathbb{P}^3$, $\beta = [L] \in \mathbb{P}^3$, where $L$ is a line. Then $d=4$.

(1). I=(4),
\[\sum_{n}{c_n^I q^n} \approx \frac{64q}{(1+q)^6} \Big{(}  8+49q+ 124q^2 +3q^3 +1040q^4+3q^5 + 124q^6+49q^7+8q^8 \Big{)}.\]

(2). I=(2,1),
\begin{multline*}
\sum_{n}{c_n^I q^n} \approx \frac{16q}{(1+q)^6} \Big{(}  14+91q+ 256q^2 +129q^3 \\
+2300q^4+129q^5 + 256q^6+91q^7+14q^8 \Big{)}.
\end{multline*}

(3). I=(0,2),
\begin{multline*}
\sum_{n}{c_n^I q^n} \approx \frac{2q}{(1+q)^6} \Big{(}  49+352q+ 1132q^2 +1184q^3 \\
+10310q^4+1184q^5 + 1132q^6+352q^7+49q^8 \Big{)}.
\end{multline*}

(4). I=(1,0,1),
\[\sum_{I}{c_n^I q^n} \approx \frac{16q}{(1+q)^6} \Big{(}  3+19q+ 49q^2 -7q^3 +720q^4-7q^5 + 49q^6+19q^7+3q^8 \Big{)}.\]

(5). I=(0,0,0,1),
\[\sum_{n}{c_n^I q^n} \approx \frac{2q}{(1+q)^6} \Big{(}  3+32q+ 92q^2 -32q^3 +1410q^4-32q^5 + 92q^6+32q^7+3q^8 \Big{)}.\]
All rational functions satisfy the functional equation
\[
 f(1/q) = q^{-4}f(q).
\]

\subsection{More complicated cases} 
We study examples when $\beta$ is not irreducible. The following two cases are typical.

\bigskip
\noindent (d). $X= \mathbb{P}^1 \times \mathbb{P}^2$, $\beta = 2([\mathbb{P}^1]\times \textup{pt})$. Then $d = 4$.

(1). I=(4),
\begin{multline*}
\sum_{n}{c_n^Iq^n} \approx \frac{9q^2}{(1+q)^8(1-q)^6}\Big(27 +22 q -535  q^2 +10716 q^3 -24861  q^4 -63414 q^5 +258841 q^6 \\
  -423032 q^7 +258841 q^8 -63414 q^9 -24861  q^{10} +10716  q^{11}  -535  q^{12} +22 q^{13} +27 q^{14}  \Big).
\end{multline*}

(2). I=(2,1),
\begin{multline*}
\sum_{n}{c_n^Iq^n} \approx \frac{18q^2}{(1+q)^8(1-q)^6}\Big(9 +14 q -95  q^2 +2520 q^3 -6321  q^4 -15246 q^5 +64775 q^6 \\
 -106672 q^7 +64775 q^8 -15246 q^9 -6321  q^{10} +2520  q^{11}  -95  q^{12} +14 q^{13} +9 q^{14}  \Big).
\end{multline*}

(3). I=(0,2),
\begin{multline*}
\sum_{n}{c_n^Iq^n} \approx \frac{18q^2}{(1+q)^8(1-q)^6}\Big(5 +11 q -25  q^2 +1182 q^3 -3219  q^4 -7323 q^5 +32423 q^6 \\
 -53788 q^7 +32423 q^8 -7323 q^9 -3219  q^{10} +1182  q^{11}  -25  q^{12} +11 q^{13} +5 q^{14}  \Big).
\end{multline*}

(4). I=(1,0,1),
\begin{multline*}
\sum_{n}{c_n^Iq^n} \approx \frac{18q^2}{(1+q)^8(1-q)^6}\Big(3 +4 q -39  q^2 +856 q^3 -2077  q^4 -5124 q^5 +21569 q^6 \\
 -35504 q^7 +21569 q^8 -5124 q^9 -2077  q^{10} +856  q^{11}  -39  q^{12} +4 q^{13} +3 q^{14}  \Big).
\end{multline*}

(5). I=(0,0,0,1),
\begin{multline*}
\sum_{n}{c_n^Iq^n} \approx \frac{9q^2}{(1+q)^8(1-q)^6}\Big(1 +4 q -53  q^2 +488 q^3 -935  q^4 -2820 q^5 +10715 q^6 \\
-17360 q^7 +10715 q^8 -2820 q^9 -935  q^{10} +488  q^{11}  -53  q^{12} +4 q^{13} + q^{14}  \Big).
\end{multline*}
All rational functions satisfy the functional equation
\[
 f(1/q) = q^{-4}f(q).
\]

\bigskip
\noindent (e). $X= \mathbb{P}^1\times \mathbb{P}^1 \times \mathbb{P}^1$, $\beta =[\mathbb{P}^1]\times \textup{pt} \times \textup{pt} + \textup{pt} \times [\mathbb{P}^1] \times \textup{pt}$. Then $d=4$.

(1). I=(4),
\begin{multline*}
\sum_{I}{c_n^I q^n} \approx \frac{128q}{(1+q)^8} \Big{(}  4+32q+113q^2 -62q^3 +515q^4 \\
+3052q^5 +515q^6-62q^7+113q^8+32q^9+4q^{10} \Big{)}.
\end{multline*}

(2). I=(2,1),
\begin{multline*}
\sum_{n}{c_n^I q^n} \approx \frac{32q}{(1+q)^8} \Big{(}  7+59q+224q^2 +10q^3 +1529q^4 \\
+6838q^5 +1529q^6+10q^7+224q^8+59q^9+7q^{10} \Big{)}.
\end{multline*}

(3). I=(0,2),
\begin{multline*}
\sum_{n}{c_n^I q^n} \approx \frac{48q}{(1+q)^8} \Big{(}  2+18q+75q^2 +54q^3 +691q^4 \\
+2544q^5 +691q^6+54q^7+75q^8+18q^9+2q^{10} \Big{)}.
\end{multline*}

(4). I=(1,0,1),
\begin{multline*}
\sum_{n}{c_n^I q^n} \approx \frac{8q}{(1+q)^8} \Big{(}  7+62q+271q^2 +68q^3 +2298q^4 \\
+9500q^5 +2298q^6+68q^7+271q^8+62q^9+7q^{10} \Big{)}.
\end{multline*} 

(5). I=(0,0,0,1),
\begin{multline*}
\sum_{n}{c_n^I q^n} \approx \frac{8q}{(1+q)^8} \Big{(}  1+8q+67q^2 +32q^3 +540q^4 \\
+2288q^5 +540q^6+32q^7+67q^8+8q^9+q^{10} \Big{)}.
\end{multline*} 
All rational functions satisfy the functional equation
\[
 f(1/q) = q^{-4}f(q).
\]

\subsection{Denominators}
According to the numerical data above, we also conjecture the following statement about the poles of the rational functions.

\begin{conj} \label{pole}
For a nonsingular projective 3-fold $X$ and nonzero $\beta \in H_2(X ,\mathbb{Z})$, all the partition functions of Chern numbers \footnote{The number of partition functions is equal to the number of partitions of the virtual dimension  $\int_{\beta}{c_1(X)}$.}
\[
\sum_{n}{c_n^Iq^n}
\]
are the Laurent expansion of rational functions which have the same poles.
\end{conj}

Note that similarly as Conjecture \ref{mainconj} (2) (see Remark \ref{mystery}), Conjecture \ref{pole} is only expected to hold for virtual Chern numbers (not for descendents). We still do not know how to attack these problems from the property of descendents by Theorem \ref{thmdescendents}. New approaches might be needed.

\nocite{*} 
\bibliographystyle{plain}
\bibliography{mybib}
\end{document}